\documentclass[a4paper,12pt,]{article}
\usepackage{hyperref}
\usepackage{amsfonts}
\usepackage{amsmath}
\usepackage{amssymb}
\usepackage{amsthm}
\usepackage{enumerate}
\usepackage[mathscr]{eucal}
\usepackage{eqlist}
\usepackage{graphicx}
\usepackage{esint}
\usepackage{color}
\usepackage{booktabs} 
\usepackage{multirow} 

\usepackage[body={15.6true cm, 24.9true cm}]{geometry}

\newtheorem{theorem}{\hskip\parindent\bf Theorem}[section]
\newtheorem{lemma}{\hskip\parindent\bf Lemma}[section]
\newtheorem{remark}{\hskip\parindent\bf Remark}[section]
\newtheorem{definition}{\hskip\parindent\bf Definition}[section]
\numberwithin{equation}{section}
\numberwithin{equation}{section} \allowdisplaybreaks
\usepackage{mathdots}

\renewcommand\abstract{{\bf Abstract}}

\begin{document}

\title {\LARGE \bf A Critical Case for the Near Field Refraction Problem in Negative Refractive Index Material}

\author{Yi He$^a$,~~Feida Jiang$^{a,b,}$\thanks{Corresponding author.
E-mail address: jiangfeida@seu.edu.cn},~~Haokun Sui$^a$}
\maketitle
\begin{center}
\begin{minipage}{12cm}
\begin{description}
\item \small
 $a$.~School of Mathematics and Shing-Tung Yau Center of Southeast University, Southeast University, Nanjing 211189, P.R. China
 \item \small
$b$.~Shanghai Institute for Mathematics and Interdisciplinary Sciences, Shanghai 200433, P.R. China
\end{description}
\end{minipage}
\end{center}


	

\begin{abstract}{\bf:}{\footnotesize
~This paper studies the near field refraction problem for the critical case in negative refractive index material. Snell law in vector form is first reviewed, and the physical constraint governing refraction is discussed. We then define the refractor for the critical case and analyze several of its key properties. Building on these properties, we introduce the notion of a weak solution to this problem and prove the existence of the weak solution when the target measure is either discrete or a finite Radon measure. The proof relies on the Minkowski method, an iterative approach to refraction and reflection problems in geometric optics.  }		
\end{abstract}
	
	

\tableofcontents

\section{Introduction}
\label{sec1}

\sloppy{}

Negative refraction is a unique electromagnetic phenomenon whose theoretical foundations can be traced back to Veselago in 1968 \cite{Ve68}. In conventional optical media, the incident and refracted rays lie on opposite sides of the interface normal within the plane of incidence. In contrast, when an electromagnetic wave passes from a positive refractive index material into a negative refractive index material, the incident and refracted rays lie on the same side of the interface normal, giving rise to the phenomenon known as negative refraction. Although this phenomenon had been theoretically predicted, direct experimental verification remained elusive for more than three decades. In 2000, Smith et al.\cite{SP00} developed an artificial composite metamaterial that exhibited simultaneously negative effective permittivity and effective permeability in the microwave regime, bringing negative refraction closer to experimental vrealization. In 2001, Shelby et al. \cite{SS01} subsequently observed negative refraction directly through a prism experiment using such artificial materials. Since the beginning of the 21st century, the distinctive optical responses of negative refractive index materials have motivated research into various applications, including optical device design\cite{CT05}, lens imaging\cite{PR02}\cite{Pen00}, directional control of electromagnetic radiation\cite{ET02} and the development of novel optical devices\cite{LH04}.

Near field refraction problem has attracted considerable attention at the intersection of geometric optics and partial differential equations. Near field refraction problem is formulated as follows: given a point source in medium \textsc{I} and a target region in medium \textsc{II} at a finite distance from the source, one seeks an interface $\mathcal{R}$ separating the two media such that the emitted rays are refracted onto a specific point in the target region. In 2014, Gutiérrez and Huang\cite{GH14} studied the point source near field refraction problem using a framework based on continuous mappings and measure equations, and established the existence of weak solution under suitable assumptions. Other approaches have also been developed to address the existence of the weak solution to the near field refraction problem. In 2021, Liu and Wang\cite{LW21} showed that both point sources and parallel sources near field refraction problem without loss of energy can be formulated as nonlinear optimization problems, thereby providing an alternative approach to proving the existence of the weak solution.

Besides existence of the weak solution, the regularity of weak solution and the numerical solution of the near field refraction problem are also important problems in the study of near field refraction problem. In 2015, Gutiérrez and Tournier\cite{GT15} established local $C^{1,\alpha}$ estimates for parallel sources near field reflection and refraction problem, yielding local regularity results for the corresponding weak solutions. In 2021, Gutiérrez and Tournier\cite{GT21} extended local $C^{1,\alpha}$ estimates to the point source near field refraction problem. In the same year, Gutiérrez and Mawi\cite{GM21} developed a numerical method for the point source near field refraction problem and established that the algorithm terminates after finitely many iterations for a prescribed accuracy. Further results on near field refraction problem can be found in\cite{Ji23}\cite{JT14}\cite{JT18}\cite{Tr21}. It is worth noting that the aforementioned studies primarily concern media with positive refractive indices. With the development of negative refractive index material, researchers have increasingly turned their attention to near field refraction problem in negative refractive index material.

For refraction problem in negative refractive index material, Gutiérrez and Stachura\cite{GS15} derived a form of Snell law suitable for negative refractive index material and extending the near field refraction problem from positive refractive index media to negative refractive index media. In 2016, Gutiérrez and Stachura\cite{GS16} further investigated near field refraction involving two interfaces in negative refractive index media. In 2017, Stachura\cite{St17} employed a Minkowski method to study the existence of weak solutions to the near field refraction problem without loss of energy in negative refractive index material. In 2026, Jiang and Sui \cite{JS26}, the second and third authors of this paper, established the existence of weak solutions to the near field refraction problem with loss of energy loss in negative refraction material.These studies indicate that the geometric construction of refracting surfaces depends strongly on the relative refractive index of the two media, and that different parameter regimes require corresponding geometric analyses.

In this paper, we consider the following problem: Suppose that $\Omega$ is a domain in $S^{n-1}$ and $D$ is a set contained in a $n-1$  dimensional hypersurface, $f$ and $g$ are two integrable functions on $\Omega$ and $D$ respectively, that is, $f\in L^{1}(\Omega)$, $g\in L^{1}(D)$. Consider two homogeneous, isotropic media: medium \textsc{I} and medium \textsc{II}, surrounded by  $\Omega$ and $D$ respectively which have different optical densities. Given a point $P\in D$, we want to construct a surface $\mathcal{R}$ separating media \textsc{I} and \textsc{II}, such that all rays emitted from the original point $O\in \Omega$ in medium \textsc{I} with intensity $f(x)$, $x\in \Omega$, are refracted by $\mathcal{R}$ to a specific point $P \in D$ in medium \textsc{II} and the corresponding illumination intensity is $g(P)$. Assuming that the refractive index of medium \textsc{I} is $n_{1}>0$, the refractive index of medium \textsc{II} is $n_{2}<0$, and set the relative refractive index $\kappa= \dfrac{n_{2}}{n_{1}}$, so $\kappa<0$. Notice that in application, it is natural to study the refraction problem in $S^{2}$, see Figure \ref{fig1}. However, in this paper, we directly study the problem in $S^{n-1}$ with $n \geq 2$ for its generality. Previous researches only consider the cases when $\kappa < -1$ and $-1 < \kappa < 0$, that is, the refracted light lying close to the normal and departing from the normal. In positive refractive index material, if $\kappa = 1$, then the refraction angle equals incident angle, hence the refraction phenomenon doesnot occur. However, in negative refractive index material, if we set $n_{2} = -n_{1}$, then $\kappa = -1$. In this case, the refraction angle also equals incident angle. Since the direction of light propagation has changed, the negative refraction phenomenon also occurs, see Figure \ref{fig2}. So this case can be seen as a critical case in  negative refraction problem. This paper will prove the existence of the weak solution to the near field refraction problem for the case $\kappa = -1$.

\begin{figure}[h]
  \centering
  \includegraphics[width=8.3cm]{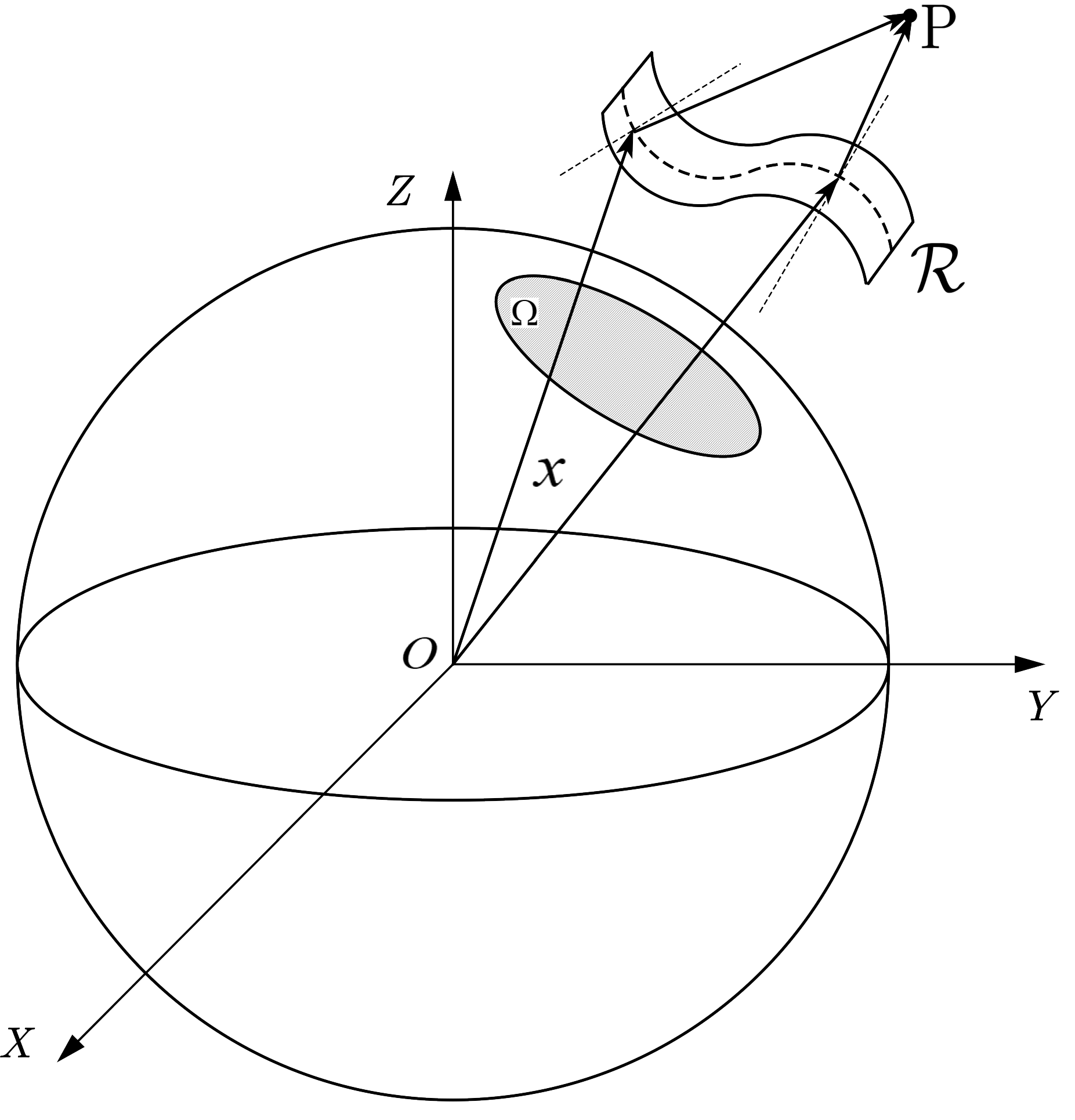}
\caption{Statement of the near field refraction problem in $S^{2}$.}\label{fig1}
\end{figure} 

\begin{figure}[h]
  \centering
  \includegraphics[width=8.3cm]{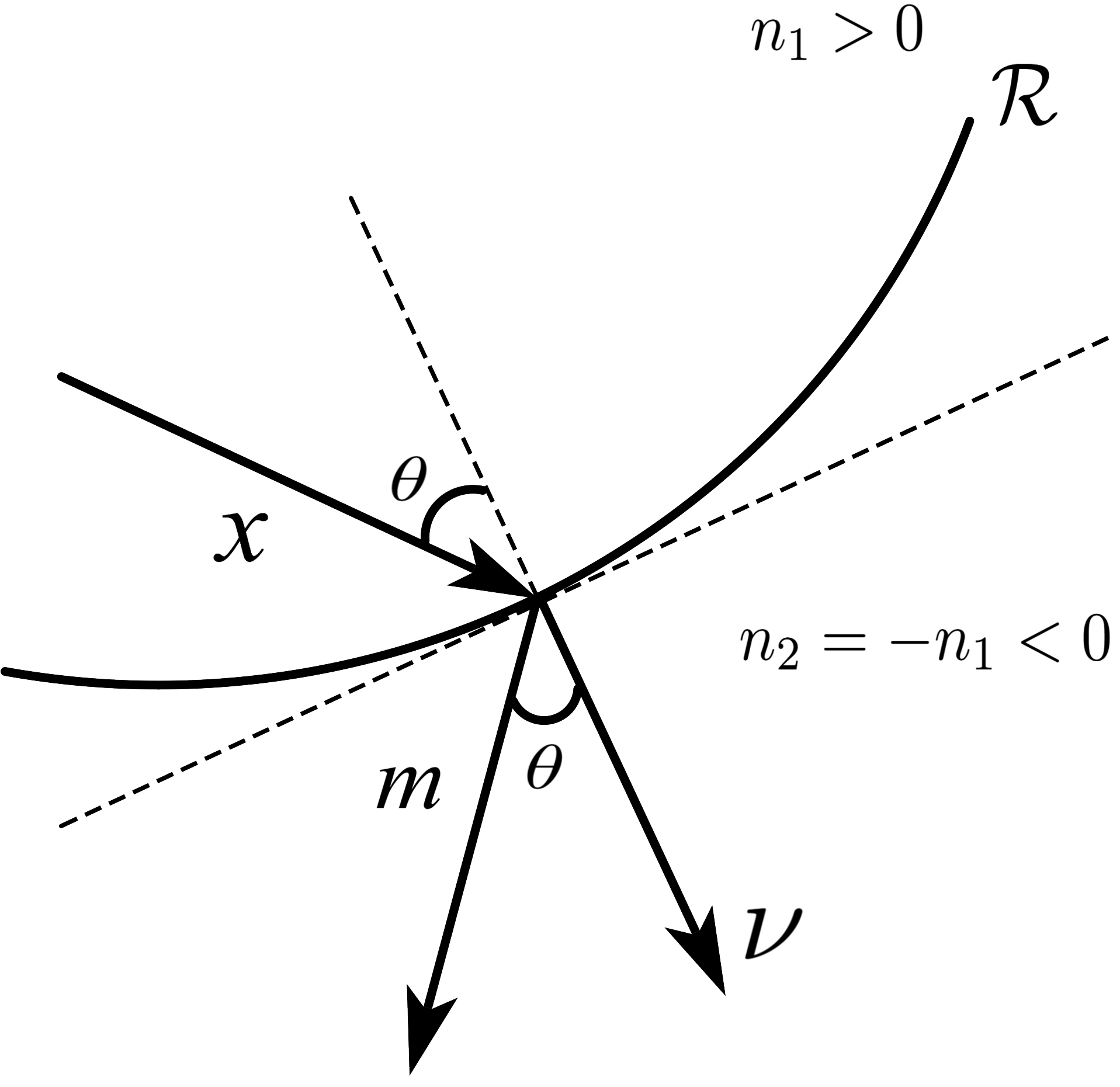}
\caption{Sketch of the refraction when $\kappa = -1$.}\label{fig2}
\end{figure} 

Now we state the main results of this paper.

\begin{theorem}
  Suppose that for any $x \in \bar{\Omega}$ and $P \in \bar{D}$, there holds $x \cdot \dfrac{P}{\vert P \vert} > -1$. Also suppose that $f \in L^{1}(\bar{\Omega})$ with $\inf\limits_{\bar{\Omega}}f > 0$, and for $j = 2,\ldots,m$, the distance between $P_{1}$ and $P_{j}$ is sufficient large. Let $\mu$ be a Radon measure defined on $\bar{D}$ and the condition $\displaystyle\int_{\bar{\Omega}}f(x)dx = \mu(\bar{D})$ holds.
   Then the weak solution to the near field refraction problem with emitting illumination intensity $f(x)$ and prescribe refracted illumination intensity $\mu$ for the case $\kappa = -1$ exists.
  \label{thm1}
\end{theorem}

In order to prove Theorem \ref{thm1}, we first give Snell law in vector form and discuss the physical constraints of $\bar{\Omega}$ and $\bar{D}$ to ensure that total internal reflection cannot occur. Then we introduce the refractor for the case $\kappa = -1$ and give some properties of it. Next, the existence results when the underlying measure is the finite sum of $\delta$-measures is proved by using approximation by ovals, see Theorem \ref{thm14}. At last, using these results, we can prove the main result of this paper.

The rest of the content is organized as follows: In Section \ref{sec2}, we give the Snell law in vector form in negative refractive index material, which is the preliminary of this paper. In Section \ref{sec3}, we give the definition of the refractor for the case $\kappa = -1$ and discuss some properties. The definition of the weak solution to the near field refraction problem for the critical case is given in Section \ref{sec4}, and the existence of the weak solution in both in discrete and general situations are proved in Section \ref{sec5}.

\section{Snell law in vector form}
\label{sec2}

\sloppy{}

In this section, we briefly introduce Snell law in vector form in negative refractive index material.

Suppose $\mathcal{R}$ is a surface in $\mathbb{R}^{n}$ that separates two homogeneous and isotropic media \textsc{I} and \textsc{II}, with refractive indices $n_{1}>0$ and $n_{2}<0$. A ray of light emitted from $O\in S^{n-1}$ in medium \textsc{I} with direction $x\in S^{n-1}$ strikes at $\Gamma$ at the point $Q$, then the refracted ray has the direction $m\in S^{n-1}$ in medium \textsc{II}. Let $\nu$ be the unit normal to $\Gamma$ at $Q$ going towards medium \textsc{II}, $\theta_{1}$ denote the angle between $x$ and $\nu$, referred to as the angle of incidence; similarly, $\theta_{2}$ denote the angle between 
$m$ and $\nu$, defined as the angle of refraction. Then we have the well-known Snell law in scalar form:
\begin{equation}\label{1}
  n_{1}\sin\theta_{1} = n_{2}\sin\theta_{2}.
\end{equation}
This law can be written in vector form as:
\begin{equation}\label{2}
  n_{1}(x\times\nu) = n_{2}(m\times\nu).
\end{equation}

From (\ref{2}), it is easily seen that $x,m$ and $\nu$ are in the same plane. Since we have set $\kappa=\dfrac{n_{2}}{n_{1}}$, (\ref{2}) can be written as
\begin{equation}\label{3}
  x-\kappa m = \lambda \nu,
\end{equation}
where $\lambda\in \mathbb{R}$ is given by
\begin{equation}\label{4}
\begin{aligned}
  \lambda  & = x\cdot \nu + \sqrt{(x\cdot \nu)^{2} - (1 - \kappa^{2})} \\
  & = x\cdot \nu + \vert \kappa \vert \sqrt{1 - \kappa^{-2}(1 - (x\cdot \nu)^{2})}.
  \end{aligned}
\end{equation}
Defining
\begin{equation}\label{5}
  \Phi(t) = t + \vert \kappa \vert \sqrt{1 - \kappa^{-2}(1 - t^{2})},
\end{equation}
then 
$$\lambda = \Phi (x \cdot \nu)$$.

Furthermore, we need to determine the physical constraints of $\bar{\Omega}$ and $\bar{\Omega}^{\ast}$ for the case $\kappa = -1$ to ensure that total internal reflection cannot occur.

When $\kappa = -1$, we know that there holds $\theta_{1} = \theta_{2}$. Set $\theta_{1} = \theta_{2} = \theta$, then in order to ensure that total internal reflection cannot occur, we only need $0 < \theta < \dfrac{\pi}{2}$. For we have $x \cdot m = \cos(\theta_{1} + \theta_{2}) = \cos(2\theta)$, so we have
\begin{equation}\label{6}
  x \cdot m > -1.
\end{equation}

\begin{remark}
Suppose that the refractive indices of media \textsc{I} and \textsc{II} are given by $n_{1}>0$ and $n_{2}<0$, and set $\kappa = \dfrac{n_{2}}{n_{1}} = -1$. Then a light ray in medium \textsc{I} in the direction $x \in S^{n-1}$ is refracted by some surface into a light ray in medium \textsc{II} in the direction $m \in S^{n-1}$ if and only if $x\cdot m > -1$.
\label{rem2}
\end{remark}

\section{Refractor and its properties}
\label{sec3}

\sloppy{}

In this section, we give the definition of the refractor for the case $\kappa = -1$ and discuss some properties of it.

From \cite{GS15}, we know that in the case $\kappa = -1$, the refracting surface is a semi-hyperboloid given by
$$\mathcal{E}(P,b) := \{h(x,P,b)x;~x \in S^{n-1},~ x \cdot P > b\},$$
where $h(x,P,b) = \dfrac{b^{2} - \vert P \vert^{2}}{2(b - x \cdot P)}$ with $\vert b \vert \leq \vert P \vert$.

Now we can define the notion of the refractor for the case $\kappa = -1$.

\begin{definition}
  Let $\mathcal{R} = \{\rho(x)x; x \in \bar{\Omega}\}$ be a surface. Then $\mathcal{R}$ is called a near field refractor for the case $\kappa = -1$ if for any $\rho(x_{0})x_{0} \in \mathcal{R}$, there exist $P \in \bar{D}$ and a number $b$ satisfies $\kappa \vert P \vert < b < \vert P \vert$, such that the refracting oval $\mathcal{E}(P,b)$ supports $\mathcal{R}$ at $\rho(x_{0})x_{0}$. Namely, for any $ x \in \bar{\Omega}$, we have $\rho(x) \leq h(x,P,b)$ and $\rho(x_{0}) = h(x_{0},P,b)$. 
  \label{def3}
\end{definition}

The following Figure \ref{fig3} shows the refracting semi-hyperboloid which refracts all ray emitted from the source $O$ to a specific point $P$ for the case $\kappa = -1$.

\begin{figure}[h]
  \centering
  \includegraphics[width=7.3cm]{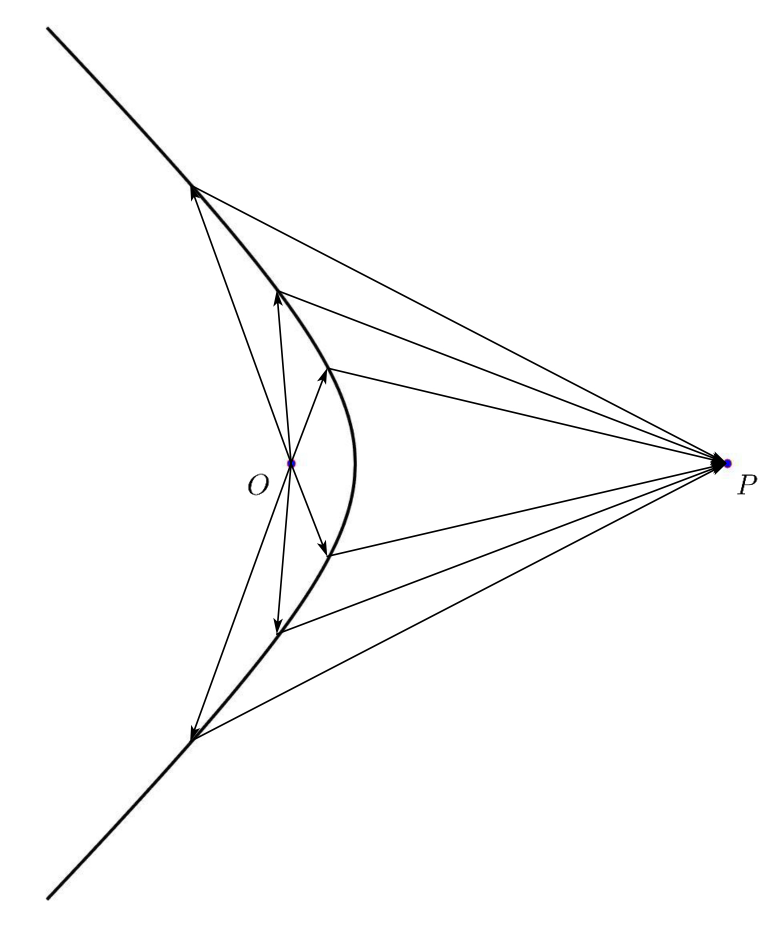}
\caption{Refracting semi-hyperboloid when $\kappa = -1$, where $O$ and $P$ are the focus of hyperboloid.}\label{fig3}
\end{figure}

Now we turn to discuss some properties of the refractor.

\begin{lemma}
  Suppose that $\mathcal{R} = \{\rho(x)x;~x\in \bar{\Omega}\}$ is a near field refractor for the case $\kappa = -1$, then $\rho$ is Lipschitz continuous with the Lipschitz constant only depending on $b$ and $\vert P \vert$.
  \label{lem4}
\end{lemma}

\begin{proof}
  Suppose that $\mathcal{E}(P,b)$ supports $\mathcal{R}$ at $\rho(x)x$, then
  $$\rho(x) = h(x,P,b) \quad \text{and} \quad \rho(y) \leq h(y,P,b)~\text{for all} ~y \in \bar{\Omega}.$$
  So we have
  \begin{equation*}
    \begin{aligned}
    \vert \rho(y) - \rho(x) \vert & \leq \vert h(y,P,b) - h(x,P,b) \vert \\
    & \leq \frac{1}{2} \left\vert \frac{b^{2} - \vert P \vert^{2}}{b - y \cdot P} - \frac{b^{2} - \vert P \vert^{2}}{b - x \cdot P} \right\vert \\
    & \leq \frac{1}{2} \left\vert \frac{(b + \vert P \vert)(b - \vert P \vert)\vert P \vert}{(b - \vert P \vert)^{2}} \right\vert \Vert x - y \Vert \\
    & \leq \frac{1}{2} \left\vert \frac{(b + \vert P \vert)\vert P \vert}{(b - \vert P \vert)}\right\vert \Vert x - y \Vert.
    \end{aligned}
  \end{equation*}
  Exchanging the roles of $x$ and $y$, then we can get $\vert \rho(x) - \rho(y) \vert \leq L \Vert x - y \Vert$ for $L = \dfrac{1}{2} \left\vert \dfrac{(b + \vert P \vert)\vert P \vert}{(b - \vert P \vert)}\right\vert > 0$. Hence $\rho$ is Lipschitz continuous on $\bar{\Omega}$.
  
\end{proof}

Next, we define refractor mapping and trace mapping, and discuss some properties of them.

\begin{definition}

Suppose the near field refractor $\mathcal{R} = \{\rho(x)x; x \in \bar{\Omega}\}$ is given, the refractor mapping of $\mathcal{R}$ is a multi-value map define by
  \begin{equation}\label{7}
    \mathcal{N}_{\mathcal{R}}(x) := \{P \in \bar{D};~ \mathcal{E}(P,b)~ supports~ \mathcal{R} ~at ~ \rho(x)x\}.
  \end{equation}
  Given $P \in \bar{D}$, the trace mapping of $\mathcal{R}$ is defined by
  \begin{equation}\label{8}
    \mathcal{T}_{\mathcal{R}}(P) := \{x \in \bar{\Omega};~ P \in \mathcal{N}_{\mathcal{R}}(x)\}.
  \end{equation}

\label{def5}
\end{definition}

\begin{lemma}
Suppose that $P \in \bar{D}$, then $\mathcal{T}_{\mathcal{R}}(P)$ is a closed set in $\bar{\Omega}$.
\label{lem6}
\end{lemma}

\begin{proof}
  Let $x_{n} \in \mathcal{T}_{\mathcal{R}}(P)$ with $x_{n} \rightarrow x_{0}$, we need to prove that $x_{0} \in \mathcal{T}_{\mathcal{R}}(P)$.
  
  Because $x_{n} \in \mathcal{T}_{\mathcal{R}}(P)$, then there exists $b > 0$, such that $\mathcal{E}(P,b)$ supports $\mathcal{R}$ at $\rho(x_{n})x_{n}$, then we have $\rho(x_{n}) = h(x_{n},P,b)$. For $x_{0} \in \bar{\Omega}$, Lemma \ref{lem4} implies that $\rho(x_{n})x_{n} \rightarrow \rho(x_{0})x_{0}$, from which we deduce $\rho(x_{0}) = h(x_{0},P,b)$. From Definition \ref{def3}, we have $\rho(x) \leq h(x,P,b)$, hence $\rho(x) \leq h(x_{0},P,b)$. So $\mathcal{E}(P,b)$ supports $\mathcal{R}$ at $\rho(x_{0})x_{0}$, hence $x_{0} \in \mathcal{T}_{\mathcal{R}}(P)$.
\end{proof}

\begin{lemma}
 For any $F \subset \bar{D}$, we have
  \newline (a) $[\mathcal{T}_{\mathcal{R}}(F)]^{c} \subset \mathcal{T}_{\mathcal{R}}(F^{c})$;
  \newline (b) The set $\mathcal{M} = \{F \subset \bar{D};~\mathcal{T}_{\mathcal{R}}(F)~is~Lebesgue~measurable\}$ is a $\sigma$-algebra containing all Borel sets in $\bar{D}$.
\label{lem7}
\end{lemma}

\begin{proof}
  $(a)$ If $x \in [\mathcal{T}_{\mathcal{R}}(F)]^{c}$, then $\mathcal{N}_{\mathcal{R}}(x) \cap F = \emptyset$, then $\mathcal{N}_{\mathcal{R}}(x) \cap F^{c} \neq \emptyset$. Then by the definition of trace mapping (\ref{8}), we have $x \in \mathcal{T}_{\mathcal{R}}(F^{c})$.
  
  $(b)$ First, we need to prove that the set $\mathcal{M}$ is a $\sigma$-algebra.
  
  Obviously, we have $\mathcal{T}_{\mathcal{R}}(\emptyset) = \emptyset$, $\mathcal{T}_{\mathcal{R}}(\bar{D}) = \bar{\Omega}$.
  For $F_{i} \in \mathcal{M}$, we have
  \begin{equation*}
      \begin{aligned}
          \mathcal{T}_{\mathcal{R}}(\bigcup\limits_{i=1}^{\infty}F_{i}) & = \bigcup\limits_{P \in \bigcup\limits_{i=1}^{\infty}F_{i}}\{x \in \bar{\Omega};~P\in \mathcal{N}_{\mathcal{R}}(x)\} \\
          & = \bigcup\limits_{i=1}^{\infty}\bigcup\limits_{P \in F_{i}}\{x \in \bar{\Omega};~P\in \mathcal{N}_{\mathcal{R}}(x)\} =\bigcup\limits_{i=1}^{\infty} \mathcal{T}_{\mathcal{R}}(F_{i}),
      \end{aligned}
  \end{equation*}
  so $\mathcal{M}$ is closed under countable union. Also for $F \in \mathcal{M}$, using $(a)$, we can get
  \begin{equation*}
    \begin{aligned}
    \mathcal{T}_{\mathcal{R}}(F^{c}) & = \{x \in \bar{\Omega};~\mathcal{N}_{\mathcal{R}}(x) \cap F^{c} \neq \emptyset\} \\
    & = \{x \in \bar{\Omega};~\mathcal{N}_{\mathcal{R}}(x) \cap F = \emptyset\} \\ 
    & \cup \{x \in \bar{\Omega};~\mathcal{N}_{\mathcal{R}}(x) \cap F^{c} \neq \emptyset,~\mathcal{N}_{\mathcal{R}}(x) \cap F \neq \emptyset\} \\
    & = [\mathcal{T}_{\mathcal{R}}(F)]^{c} \cup [\mathcal{T}_{\mathcal{R}}(F^{c}) \cap \mathcal{T}_{\mathcal{R}}(F)].
    \end{aligned}
  \end{equation*}
  Since $\vert \mathcal{T}_{\mathcal{R}}(F^{c}) \cap \mathcal{T}_{\mathcal{R}}(F) \vert = 0$ and $\mathcal{T}_{\mathcal{R}}(F)$ is measurable, then $\mathcal{T}_{\mathcal{R}}(F^{c})$ is measurable, hence $\mathcal{M}$ is closed under taking complements. So we have the set $\mathcal{M}$ is a $\sigma$-algebra.
  
  Next, we need to prove that $\mathcal{M}$ contains all Borel sets in $\bar{D}$.
  
  Indeed, we need to prove that for any compact set $K \subset \bar{D}$, $\mathcal{T}_{\mathcal{R}}(K)$ is compact.
  
  Choose a compact set $K \subset \bar{D}$ and take $x_{j} \in \mathcal{T}_{\mathcal{R}}(K)$ and $P_{j} \in \mathcal{N}_{\mathcal{R}}(x_{j}) \cap K$. Suppose that $\mathcal{E}(P_{j},b_{j})$ supports $\mathcal{R}$ at $\rho(x_{j})x_{j}$, then we have
  \begin{equation}\label{9}
    \rho(x_{j}) = \frac{b_{j}^{2} - \vert P_{j} \vert^{2}}{2(b_{j} - x_{j} \cdot P_{j})}
  \end{equation}
  and
  \begin{equation}\label{10}
    \rho(x) \leq \frac{b_{j}^{2} - \vert P_{j} \vert^{2}}{2(b_{j} - x \cdot P_{j})}.
  \end{equation}
  Assume that there exist $0 < C_{1} \leq C_{2}$, such that $C_{1} \leq \rho(x) \leq C_{2}$. Then on the one hand, from (\ref{9}), we have
  \begin{equation*}
    b_{j}^{2} - \vert P_{j} \vert^{2} = 2\rho(x_{j})(b_{j} - x_{j} \cdot P_{j}) \geq 2\rho(x_{j})(b_{j} - P_{j}),
  \end{equation*}
  hence we have 
  $$b_{j} \leq 2\rho(x_{j}) - \vert P_{j} \vert \leq 2C_{2} - \vert P_{j} \vert.$$
  On the other hand, from (\ref{10}), we have
  $$2\rho(x)(x \cdot P_{j} - b_{j}) \leq \vert P_{j} \vert^{2} - b_{j}^{2} \leq \vert P_{j} \vert^{2},$$
  hence we have
  \begin{equation*}
    b_{j} \geq x \cdot P_{j} - \frac{\vert P_{j} \vert^{2}}{2\rho(x)} \geq x \cdot P_{j} - \frac{\vert P_{j} \vert^{2}}{2C_{1}}.
  \end{equation*} 
  So we have
  \begin{equation}\label{11}
    x \cdot P_{j} - \frac{\vert P_{j} \vert^{2}}{2C_{1}} \leq b_{j} \leq 2C_{2} - \vert P_{j} \vert.
  \end{equation}
  Then by taking a subsequence if necessary, we may assume that $x_{j} \rightarrow x_{0}$, $P_{j} \rightarrow P_{0} \in K$ and $b_{j} \rightarrow b_{0}$ as $j \rightarrow \infty$. Taking the limit as $j \rightarrow \infty$ in (\ref{10}), we have that $\mathcal{E}(P,b)$ supports $\mathcal{R}$ at $\rho(x_{0})x_{0}$. Then $x_{0} \in \mathcal{T}_{\mathcal{R}}(K)$, hence $\mathcal{T}_{\mathcal{R}}(K)$ is compact.
\end{proof}

\begin{lemma}
Suppose that $\mathcal{E}(P_{k},b_{k})$ is a sequence of semi-hyperboloid with $P_{k} \rightarrow P_{0}$ and $b_{k} \rightarrow b_{0}$ as $k \rightarrow \infty$. If $z_{k} \in \mathcal{E}(P_{k},b_{k})$ with $z_{k} \rightarrow z_{0}$ as $k \rightarrow \infty$, then $z_{0} \in \mathcal{E}(P_{0},b_{0})$, and the normal $\nu_{k}(z_{k})$ to the semi-hyperboloid $\mathcal{E}(P_{k},b_{k})$ at $z_{k}$ satisfies $\nu_{k}(z_{k}) \rightarrow \nu(z_{0})$, where $\nu(z_{0})$ is the normal to the semi-hyperboloid  $\mathcal{E}(P_{0},b_{0})$ at the point $z_{0}$.
\label{lem8}
\end{lemma}

\begin{proof}
  Let $z = h(x,P_{k},b_{k})x$, $x \in S^{n-1}$, then the Cartesian coordinate of the equation of the semi-hyperboloid is $\vert z \vert - \vert z - P_{k} \vert = b_{k}$, so the normal vector at $z$ is $\nu_{k}(z) = \dfrac{z}{\vert z \vert} +  \dfrac{P_{k} - z}{\vert P_{k} - z \vert}$. So we have
  $$\nu_{k}(z_{k}) = \dfrac{z_{k}}{\vert z_{k} \vert} + \dfrac{P_{k} - z_{k}}{\vert P_{k} - z_{k} \vert} \rightarrow \dfrac{z_{0}}{\vert z_{0} \vert} + \dfrac{P_{0} - z_{0}}{\vert P_{0} - z_{0} \vert} = \nu(z_{0}).$$
\end{proof}

\begin{lemma}
 Suppose that $\mathcal{R}_{k} = \{\rho_{k}(x)x; ~x\in \bar{\Omega}\}$, $k \geq 1$, is a sequence of near field refractors, then fix $x_{0} \in \bar{\Omega}$, there exist $P_{k} \in \bar{D}$ and $b_{k}$, such that $\mathcal{E}(P_{k},b_{k})$ supports $\mathcal{R}_{k}$ at $\rho(x_{0})x_{0}$. Suppose that there exist $C_{1}, C_{2} > 0$, such that $C_{1} \leq \rho_{k}(x) \leq C_{2}$ in $\bar{\Omega}$, and $\rho_{k} \rightarrow \rho$ uniformly on $\bar{\Omega}$, then we have
  \newline (a) $\mathcal{R} := \{\rho(x)x; ~x\in \bar{\Omega}\}$ is also a near field refractor;
  \newline (b) For any compact set $K\subset \bar{D}$,
  $$\varlimsup_{k\rightarrow \infty} \mathcal{T}_{\mathcal{R}_{k}}(K) \subset \mathcal{T}_{\mathcal{R}}(K);$$
  \newline (c) For any open set $G \subset \bar{D}$,
  $$\mathcal{T}_{\mathcal{R}}(G) \subset \varliminf_{k\rightarrow \infty}\mathcal{T}_{\mathcal{R}_{k}}(G) \cup E,$$
  where $E$ is the singular set of $\mathcal{R}$.
\label{lem9}
\end{lemma}

\begin{proof}
  $(a)$ From Lemma \ref{lem9}, by taking a subsequence if necessary, we may assume that there exist $P_{0} \in \bar{D}$ and $b_{0}$, such that $P_{k} \rightarrow P_{0}$ and $b_{k} \rightarrow b_{0}$ as $k \rightarrow \infty$. Hence we have
  $$\rho(x_{0}) = \lim\limits_{k \rightarrow \infty}\rho_{k}(x_{0}) = \lim\limits_{k \rightarrow \infty}h(x_{0},P_{k},b_{k}) = h(x_{0},P_{0},b_{0})$$
  and
  $$\rho(x) = \lim\limits_{k \rightarrow \infty}\rho_{k}(x) \leq \lim\limits_{k \rightarrow \infty}h(x,P_{k},b_{k}) = h(x,P_{0},b_{0})~\text{for all} ~x\in \bar{\Omega}.$$
  So $\mathcal{E}(P_{0},b_{0})$ supports $\mathcal{R}$ at $\rho(x_{0})x_{0}$, then $\mathcal{R}$ is a near field refractor.
  
  $(b)$  Let $x_{0} \in \varlimsup\limits_{k\rightarrow \infty} \mathcal{T}_{\mathcal{R}_{k}}(K)$. Without loss of generality, we assume that $x_{0} \in \mathcal{T}_{\mathcal{R}_{k_{j}}}(K)$ for $j = 1,2,\ldots$, then there exists $P_{k_{j}} \in K$ such that 
  $$\rho_{k_{j}}(x) \leq h(x,P_{k_{j}},b_{k_{j}}) \quad \text{and} \quad \rho_{k_{j}}(x_{0}) = h(x_{0},P_{k_{j}},b_{k_{j}})$$
  and
  $$x \cdot P_{k_{j}} >b_{k_{j}}.$$
  From (\ref{11}), we know that there exist $P_{0} \in K$ and $b_{0} > 0$, such that $P_{k_{j}} \rightarrow P_{0}$ and $b_{k_{j}} \rightarrow b_{0}$ as $j \rightarrow \infty$. From the proof of $(a)$, $\mathcal{E}(P_{0},b_{0})$ supports $\mathcal{R}$ at $\rho(x_{0})x_{0}$, hence $x_{0} \in \mathcal{T}_{\mathcal{R}}(P_{0})$, so $x_{0} \in \mathcal{T}_{\mathcal{R}}(K)$.
  
  $(c)$ Since $\mathcal{T}_{\mathcal{R}}(F^{c}) = [\mathcal{T}_{\mathcal{R}}(F)]^{c} \cup [\mathcal{T}_{\mathcal{R}}(F^{c}) \cap \mathcal{T}_{\mathcal{R}}(F)]$, from $(b)$, we have
  \begin{equation}\label{12}
  \begin{aligned}
    \varlimsup_{k \rightarrow \infty}[\mathcal{T}_{\mathcal{R}_{k}}(G)]^{c} & \subset \varlimsup_{k \rightarrow \infty}([\mathcal{T}_{\mathcal{R}_{k}}(G)]^{c} \cup [\mathcal{T}_{\mathcal{R}_{k}}(G) \cap \mathcal{T}_{\mathcal{R}_{k}}(G^{c})])\\
    & \subset \varlimsup_{k \rightarrow \infty}\mathcal{T}_{\mathcal{R}_{k}}(G^{c}) \subset \mathcal{T}_{\mathcal{R}}(G^{c}) \\
    & =[\mathcal{T}_{\mathcal{R}}(G)]^{c} \cup [\mathcal{T}_{\mathcal{R}}(G^{c}) \cap \mathcal{T}_{\mathcal{R}}(G)].
    \end{aligned}
  \end{equation}
  Since $\mathcal{T}_{\mathcal{R}}(G^{c}) \cap \mathcal{T}_{\mathcal{R}}(G) \subset E$, taking complements in (\ref{12}), we get
  $$\mathcal{T}_{\mathcal{R}}(G) \subset \varliminf_{k\rightarrow \infty}\mathcal{T}_{\mathcal{R}_{k}}(G) \cup E.$$
\end{proof}

\section{Definition of the weak solution}
\label{sec4}

\sloppy{}

In this section, we give the definition of the weak solution to the near field refraction problem for the critical case in negative refractive index medium. We first give the definition of refractor measure  originated from \cite{St17}, then the definition of the weak solution is given based on refractor measure.

\begin{definition}
  Suppose $\mathcal{R}$ is a refractor from $\bar{\Omega}$ to $\bar{D}$ for the case $\kappa = -1$, $f \in L^{1}(\bar{\Omega})$ and $\inf\limits_{\bar{\Omega}}f >0$. The refractor measure associated with $\mathcal{R}$ and $f$ is defined by a set function on Borel subsets of $\bar{D}$:
  \begin{equation}\label{13}
    G_{\mathcal{R}}(F) := \int_{\mathcal{T}_{\mathcal{R}}(F)}f(x)dx,
  \end{equation}
  where $dx$ is the surface measure on $S^{n-1}$.
  \label{def10}
\end{definition}

\begin{remark}
  $G_{\mathcal{R}}(F)$ is a finite Borel measure defined on $\mathcal{M}$, where $\mathcal{M}$ is defined in Lemma \ref{lem7} $(b)$.
  \label{rem11}
\end{remark}

Now we can define the weak solution to the near field refraction problem for the critical case in negative refractive index medium.

\begin{definition}
  Suppose that $\mu$ is a Radon measure on the Borel subset of $\bar{D}$ and $f \in L^{1}(\bar{\Omega})$, a refractor $\mathcal{R}$ is a weak solution to the near field refraction problem for the case $\kappa = -1$ with emitting illumination intensity $f(x)$ and prescribe refracted illumination intensity $\mu$ if for any Borel set $\omega \subset \bar{D}$, there holds:
  \begin{equation}\label{14}
    G_{\mathcal{R}}(\omega) = \int_{\mathcal{T}_{\mathcal{R}}(\omega)}f(x)dx = \mu(\omega).
  \end{equation}
  \label{def12}
\end{definition}

\section{Existence of the weak solution}
\label{sec5}

In this section, we discuss the existence of the weak solution to the near field refraction problem for the critical case in negative refractive index medium. We first give a lemma which is crucial for proving the existence of the weak solution for the discrete case. Then we prove the existence of the weak solution under the assumption that $\mu$ equals finite sum of $\delta$-measures. Finally, the main result of this paper, that is, the existence of the weak solution when $\mu$ is a finite Radon measure is proved by applying the previous existence result when $\mu$ equals to finite sum of $\delta$-measures.

\subsection{Existence of the weak solution when $\mu$ is discrete measure}
\label{sec5.1}

In this subsection, we assume that $\mu$ equals finite sum of $\delta$-measures, hence all rays are refracted into finite directions. Based on this assumption, we establish the existence of the weak solution to the near field refraction problem for the critical case in negative refractive index medium when $\mu$ is discrete measure.

\begin{remark}
Suppose that $P_{1},P_{2},\ldots, P_{m},~m \geq 2$ are discrete points in $\bar{D}$, then for $\mathbf{b} = (b_{1},b_{2},\cdots ,b_{m}) \in \mathbb{R}^{m}$, the refractor is defined as
  \begin{equation}\label{15}
    \mathcal{R}(\mathbf{b}) = \{\rho(x)x;~x\in \bar{\Omega},~\rho(x)=\min_{1\leq j \leq m}h(x,P_{j},b_{j})\},
  \end{equation}
  where
  \begin{equation*}
    h(x,P_{j},b_{j}) = \frac{b_{j}^{2} - \vert P_{j} \vert^{2}}{2(b_{j} - x \cdot P_{j})}.
  \end{equation*}
\label{rem13}
\end{remark}

Now we show the existence of the weak solution when $\mu$ equals finite sum of $\delta$-measures. 

\begin{theorem}
  Suppose that $f \in L^{1}(\bar{\Omega})$ with $\inf\limits_{\bar{\Omega}}f > 0$, $P_{1},P_{2},\ldots, P_{m},~m \geq 2$ are discrete points in $\bar{D}$, $g_{1}, g_{2}, \ldots, g_{m}>0$ and the condition $\displaystyle\int_{\bar{\Omega}}f(x)dx = \sum\limits_{j=1}^{m}g_{j}$ holds. Then for any $b_{1}$ satisfies the condition $b_{1} \leq x \cdot P_{1} - \dfrac{\vert P_{1} \vert^{2}}{\delta}$, where $\delta > 0$ satisfies $b_{j} + \vert P_{j} \vert \geq \delta >0$, there exists a unique $(b_{2}, \ldots, b_{m})$, such that $\mathcal{R} = \{\rho(x)x;~x \in \bar{\Omega}\}$ is the weak solution to the near field refraction problem for the case $\kappa = -1$, where $\rho(x) = \min\limits_{1 \leq j \leq m}h(x,P_{j},b_{j})$.
  \label{thm14}
\end{theorem}

In order to prove Theorem \ref{thm14}, we need the following Lemma \ref{lem15}.

\begin{lemma}
Suppose that $f \in L^{1}(\bar{\Omega})$ with $\inf\limits_{\bar{\Omega}}f>0$. Let $P_{1},P_{2},\ldots, P_{m},~m \geq 2$ are discrete points in $\bar{D}$, $g_{1}, g_{2}, \ldots, g_{m}>0$. Define a set $W := \{(b_{2}, \ldots, b_{m});~-\vert P_{j} \vert < b_{j} < \vert P_{j} \vert,~j=2,\ldots,m\}$, and for any $\mathbf{b} \in W$, we have $\displaystyle\int_{\mathcal{T}_{\mathcal{R}(\mathbf{b})}(P_{j})}f(x)dx \leq g_{i}, 2 \leq i \leq m$, then
  \newline (a) $W \neq \emptyset$;
  \newline (b) If $\mathbf{b} = (b_{2}, \ldots ,b_{m}) \in W$, then for $2 \leq j \leq m$, we have 
  \begin{equation*}
    b_{j} \geq x_{0} \cdot P_{j} - \frac{\vert P_{j} \vert^{2}}{\vert P_{1} \vert + b_{1}}.
  \end{equation*}
\label{lem15}
\end{lemma}

\begin{proof}
  $(a)$ We first claim that if $\mathbf{b} = (b_{2},\ldots,b_{m}) \in W$, then $0 < g_{1} \leq G_{\mathcal{R}(\mathbf{b})}(P_{1})$.
  
  Indeed, for we have $\displaystyle\int_{\bar{\Omega}}f(x)dx = \sum\limits_{j=1}^{m}g_{j}$, then 
  \begin{equation*}
    \sum_{j=1}^{m}G_{\mathcal{R}(\mathbf{b})}(P_{j})  = \sum_{j=1}^{m}\int_{\mathcal{T}_{\mathcal{R}(\mathbf{b})}(P_{j})}f(x)dx = \int_{\Omega}f(x)dx = \sum_{j=1}^{m}g_{j}.
  \end{equation*}
  Hence we have
  \begin{equation*}
    g_{1} - G_{\mathcal{R}(\mathbf{b})}(P_{1}) + \sum_{j=2}^{m}(g_{j} - G_{\mathcal{R}(\mathbf{b})}(P_{j})) = 0.
  \end{equation*}
  Since $\mathbf{b} = (b_{2},\ldots,b_{m}) \in W$, then $g_{j} - G_{\mathcal{R}(\mathbf{b})}(P_{j}) \geq 0$ for $j = 2,\ldots,m$, hence $g_{1} - G_{\mathcal{R}(\mathbf{b})}(P_{1}) \leq 0$. So we have $g_{1} \leq G_{\mathcal{R}(\mathbf{b})}(P_{1})$.
  
  For any $b_{1}$ satisfies the condition $b_{1} \leq x \cdot P_{1} - \dfrac{\vert P_{1} \vert^{2}}{\delta}$, we then claim that $\mathcal{R}(\mathbf{b}) = \mathcal{E}(P_{1},b_{1}) = \{h(x,P_{1},b_{1})x;~x \in S^{n-1},x \cdot P_{1} > b_{1}\}.$
  
  Indeed, for we have
  \begin{equation*}
    h(x,P_{1},b_{1})  \leq \frac{\vert P_{1} \vert^{2}}{2(x \cdot P_{1} - b_{1})}  \leq \frac{\delta \vert P_{1} \vert^{2}}{2\vert P_{1} \vert^{2}} \leq \frac{b_{j} + \vert P_{j} \vert}{2} \leq h(x,P_{j},b_{j}).
  \end{equation*}
  So we have $\mathcal{R}(\mathbf{b}) = \mathcal{E}(P_{1},b_{1})$. Then for $j \neq 1$, we have $G_{\mathcal{R}(\mathbf{b})}(P_{j}) = 0 < g_{j}$, hence $W \neq \emptyset$.
  
  $(b)$ Since $\mathcal{R}(\mathbf{b})$ is a near field refractor, then there exists $x_{0} \in \bar{\Omega}$, such that $\mathcal{E}(P_{1},b_{1})$ supports $\mathcal{R}(\mathbf{b})$ at $\rho(x_{0})x_{0}$. Then we have
  $$\frac{\vert P_{1} \vert + b_{1}}{2} \leq \rho(x_{0}) \leq h(x_{0},P_{j},b_{j}) \leq \frac{\vert P_{j} \vert^{2}}{2(x_{0} \cdot \vert P_{j} \vert - b_{j})}.$$
  Hence
  $$\frac{1}{\vert P_{1} \vert + b_{1}} \geq \frac{x_{0} \cdot P_{j}}{\vert P_{j} \vert^{2}} - \frac{b_{j}}{\vert P_{j} \vert^{2}},$$
  from which we deduce
  $$ b_{j} \geq x_{0} \cdot P_{j} - \frac{\vert P_{j} \vert^{2}}{\vert P_{1} \vert + b_{1}}.$$
\end{proof}

\begin{remark}
From Theorem 2.5 in \cite{GH14}, we know that $W$ is compact.
\label{rem16}
\end{remark}

Based on the above lemma, now we prove the existence of the weak solution for the critical case when $\mu$ equals finite sum of $\delta$-measures.

\begin{proof}[Proof of Theorem \ref{thm14}]
  We consider the mapping $d:W \rightarrow \mathbb{R}:~(b_{2},\ldots,b_{m}) \mapsto b_{2} + \ldots + b_{m}$. Since $W$ is compact, $d$ can attain its minimum at some points in $W$. Suppose that $d$ attains its minimum at $(a_{2},\ldots,a_{m})$, we claim that $\mathcal{R}(\mathbf{a})$ is the weak solution of the near field refraction problem, where $\mathbf{a} = (b_{1},a_{2},\ldots,a_{m})$. In particular, $\mathcal{R}(\mathbf{a}) = \{\rho_{\mathbf{a}}(x)x;~x\in \bar{\Omega},\rho_{\mathbf{a}}(x) = \min\limits_{1 \leq j \leq m}h(x,P_{j},a_{j})\}$.
  
  Indeed, without loss of generality, we assume for contradiction that $\mathcal{R}(\mathbf{a})$ isnot the weak solution of the near field refraction problem, then $G_{\mathcal{R}(\mathbf{a})}(P_{2}) < g_{2}$. Let $\varepsilon > 0$ and define $\bar{\mathbf{a}} = (b_{1},a_{2}-\varepsilon,\ldots,a_{m})$, the corresponding semi-hyperboloid is denoted as $\mathcal{R}(\bar{\mathbf{a}})$. From the continuity of $G_{\mathcal{R}}$, there holds $G_{\mathcal{R}(\bar{\mathbf{a}})}(P_{2}) < g_{2}$ for $\varepsilon$ is sufficiently small.
  
  Now we claim that $\mathcal{T}_{\mathcal{R}(\bar{\mathbf{a}})}(P_{j}) \subset \mathcal{T}_{\mathcal{R}(\mathbf{a})}(P_{j})$ holds almost everywhere for $j \geq 3$.
  
  Indeed, let $x_{0} \in \mathcal{T}_{\mathcal{R}(\bar{\mathbf{a}})}(P_{j})$ and $x_{0}$ isnot the singular point of the surface, then $P_{j} = \mathcal{N}_{\mathcal{R}(\mathbf{a})}(x_{0})$. From the definition of the refractor mapping (\ref{7}), there exists $\mathcal{E}(P_{j},b_{j})$ supports $\mathcal{R}(\bar{\mathbf{a}})$ at $\rho_{\bar{\mathbf{a}}}(x_{0})x_{0}$. Then we have
  $$\rho_{\bar{\mathbf{a}}}(x) \leq h(x,P_{j},\bar{a}_{j}) \quad \text{and} \quad \rho_{\bar{\mathbf{a}}}(x_{0}) \leq h(x_{0},P_{j},\bar{a}_{j}).$$
  
  Now we turn to prove that $\rho_{\bar{\mathbf{a}}}(x_{0}) = h(x_{0},P_{j},a_{j})$.
  
  Indeed, assume for contradiction that there exists $i \neq j$ with $1 \leq j \leq m$, such that $\rho_{\bar{\mathbf{a}}}(x_{0}) = h(x_{0},P_{j},\bar{a}_{j}) < h(x_{0},P_{i},a_{i})$, then we have $P_{i} \in \mathcal{N}_{\mathcal{R}(\mathbf{a})}(x_{0})$, hence $P_{i} = P_{j}$. This is a contradiction, so we have $\rho_{\bar{\mathbf{a}}}(x_{0}) = h(x_{0},P_{j},a_{j})$.
  
  Then we claim that $P_{j} \in \mathcal{N}_{\mathcal{R}(a)}(x_{0})$.
  
  Indeed, from the definition of $\rho_{\mathbf{a}}(x)$, we have $\rho_{\bar{\mathbf{a}}}(x) \leq \rho_{\mathbf{a}}(x) \leq h(x,P_{j},a_{j})$. Then we have
  $$\rho_{\bar{\mathbf{a}}}(x_{0}) \leq \rho_{\mathbf{a}}(x_{0}) \leq h(x_{0},P_{j},a_{j}) = \rho_{\bar{\mathbf{a}}}(x_{0}).$$
  Hence $\rho_{\mathbf{a}}(x_{0}) = \rho_{\bar{\mathbf{a}}}(x_{0})$, so we have $P_{j} \in \mathcal{N}_{\mathcal{R}(a)}(x_{0})$.
  
  Combine with above two conclusions, we have $\mathcal{T}_{\mathcal{R}(\bar{\mathbf{a}})}(P_{j}) \subset \mathcal{T}_{\mathcal{R}(\mathbf{a})}(P_{j})$, hence $G_{\mathcal{R}(\bar{\mathbf{a}})}(P_{j}) \leq G_{\mathcal{R}(\mathbf{a})}(P_{j}) \leq g_{j}$ for $j \geq 3$, so $\bar{\mathbf{a}} \in W$. This is a contradiction with $\mathbf{a}$ is the minimum of $d$. So $\mathcal{R}(\mathbf{a})$ is the weak solution of the near field refraction problem.
\end{proof}

\subsection{Existence of the weak solution when $\mu$ is finite Radon measure}
\label{sec5.2}

In this subsection, based on Theorem \ref{thm14}, we prove the existence of the weak solution to the near field refraction problem for the critical case.

\begin{proof}[Proof of Theorem \ref{thm1}]
  In order to prove the existence of the weak solution for the critical case, we need to show that given $P_{0} \in supp(\mu)$, then for any Borel set $F \subset D$ with $P_{0} \notin F$, there exists a refractor $\mathcal{R}$, such that $\displaystyle\int_{\mathcal{T}_{\mathcal{R}}(F)}f(x)dx = \mu(F)$, where $\mu$ is the Radon measure defined on $\bar{D}$.
  
  We fix $-\vert P_{1} \vert < b_{1} \leq x\cdot P_{1} - \dfrac{\vert P_{1} \vert^{2}}{\delta}$, where $\delta$ is given in Theorem \ref{thm14}, and divide $D$ into a finite, disjoint union of Borel set $\{E_{j}\}$ with nonempty interiors. For each $E_{j}$, we have $diam(E_{j}) \leq \varepsilon$, where $\varepsilon$ is sufficient small. Suppose that $P_{0}$ belongs to the interior of some $E_{k}$. For $P_{0} \in supp(\mu)$, then $\mu(E_{k}) > 0$. We discard the sets in $E_{j}$ which makes $\mu(E_{j}) = 0$, and label the remaining sets as $F_{1}^{1},\ldots,F_{m_{1}}^{1}$. Without loss of generality, we assume that $P_{0} \in (F_{1}^{1})^{\circ}$ and $\mu(F_{j}^{1}) > 0$. Taking $P_{j}^{1} \in F_{j}^{1}$, then we have $P_{1}^{1} = P_{0}$.
  
  Defining a measure on $\bar{D}$ as $\mu_{1} = \sum\limits_{j=1}^{m_{1}}\mu(F_{j}^{1})\delta_{P_{j}^{1}}$, then we have
  $$\mu_{1}(\bar{D}) = \mu(\bar{D}) = \displaystyle\int_{\bar{\Omega}}f(x)dx.$$
  From Theorem \ref{thm14}, there exists a near field refractor $\mathcal{R}_{1} = \{\rho_{1}(x)x;~\rho_{1}(x) = \min\limits_{1 \leq j \leq m_{1}}h(x,P_{j}^{1},b_{j})\}$, such that $\mu_{1}(F) = \displaystyle\int_{\mathcal{T}_{\mathcal{R}_{1}}(F)}f(x)dx$.
  
  Next, we divide each $F_{j}^{1}$ into finite, disjoint union of Borel sets $\{E_{j}^{'}\}$ with nonempty interiors. For each $E_{j}^{'}$, we have $diam(E_{j}^{'}) \leq \dfrac{\varepsilon}{2}$, where $\varepsilon$ is sufficient small. Suppose that $P_{0}$ belongs to the interior of some $E_{k}^{'}$. For $P_{0} \in supp(\mu)$, then $\mu(E_{k}^{'}) > 0$. We discard the sets in $E_{j}^{'}$ which makes $\mu(E_{j}^{'}) = 0$, and label the remaining sets as $F_{1}^{2},\ldots,F_{m_{2}}^{2}$. Without loss of generality, we assume that $F_{1}^{2} \subset F_{1}^{1}$ and $P_{0} \in (F_{1}^{2})^{\circ}$. Taking $P_{j}^{2} \in F_{j}^{2}$, then we have $P_{1}^{2} = P_{0}$.
  
  Defining a measure on $\bar{D}$ as $\mu_{2} = \sum\limits_{j=1}^{m_{2}}\mu(F_{j}^{2})\delta_{P_{j}^{2}}$, then we have
  $$\mu_{2}(\bar{D}) = \mu(\bar{D}) = \displaystyle\int_{\bar{\Omega}}f(x)dx.$$
  From Theorem \ref{thm14}, there exists a near field refractor $\mathcal{R}_{2} = \{\rho_{2}(x)x;~\rho_{2}(x) = \min\limits_{1 \leq j \leq m_{2}}h(x,P_{j}^{2},b_{j})\}$, such that $\mu_{2}(F) = \displaystyle\int_{\mathcal{T}_{\mathcal{R}_{2}}(F)}f(x)dx$.
  
  Continuing in this way, we obtain a sequence of disjoint Borel sets $\{F_{j}^{l}\}$, where $1 \leq j \leq m_{l}$ and $diam(F_{j}^{l}) \leq \dfrac{\varepsilon}{2^{l}}$. For any $F_{j}^{l}$, we have $\mu(F_{j}^{l}) > 0$ with $P_{0} \in (F_{j}^{l})^{\circ}$ and $F_{1}^{l + 1} \subset F_{1}^{l}$, we take $P_{j}^{l} \in F_{j}^{l}$ such that $P_{1}^{l} = P_{0}$ holds for any $j$ and $l$. Then the measure $\mu_{l} = \sum\limits_{j=1}^{m_{l}}\mu(F_{j}^{l})\delta_{P_{j}^{l}}$ satisfies 
   $$\mu_{l}(\bar{D}) = \mu(\bar{D}) = \displaystyle\int_{\bar{\Omega}}f(x)dx.$$
  From Theorem \ref{thm14}, there exists a near field refractor $\mathcal{R}_{l} = \{\rho_{l}(x)x;~\rho_{l}(x) = \min\limits_{1 \leq j \leq m_{l}}h(x,P_{j}^{l},b_{j})\}$, such that $\mu_{l}(F) = \displaystyle\int_{\mathcal{T}_{\mathcal{R}_{l}}(F)}f(x)dx$.
  
 Let $x_{0},x_{1} \in \bar{\Omega}$ and suppose that $\mathcal{E}(P_{0},b_{0})$ supports $\mathcal{R}_{l}$ at $\rho(x_{0})x_{0}$. Since $x \cdot P_{0} > b_{0}$, then there exists $\alpha > 0$, such that $x \cdot P_{0} - b_{0} \geq \alpha > 0$, then we have
 \begin{equation*}
   \begin{aligned}
   \vert \rho_{l}(x_{1}) - \rho_{l}(x_{0})\vert & \leq \left\vert \frac{\vert P_{0} \vert^{2} - b_{0}^{2}}{2(x_{1} \cdot P_{0} - b_{0})} - \frac{\vert P_{0} \vert^{2} - b_{0}^{2}}{2(x_{0} \cdot P_{0} - b_{0})} \right\vert \\
   & \leq \left\vert \frac{(\vert P_{0} \vert^{2} - b_{0}^{2}) \vert P_{0} \vert}{2(x_{1} \cdot P_{0} - b_{0})(x_{0} \cdot P_{0} - b_{0})} \right\vert \Vert x_{1} - x_{0} \Vert \\
   & \leq \frac{\vert P_{0}\vert^{3}}{2 \alpha^{2}}\Vert x_{1} - x_{0} \Vert.
   \end{aligned}
 \end{equation*}
 Exchanging the roles of $x_{0}$ and $x_{1}$, then we get $\vert \rho_{l}(x_{0}) - \rho_{l}(x_{1})\vert \leq \dfrac{\vert P_{0}\vert^{3}}{2 \alpha^{2}}\Vert x_{0} - x_{1} \Vert$. Hence $\{\rho_{l};~l \geq 1\}$ is a family of  uniformly bounded and equicontinuous functions. Then from Arezlà-Ascoli Theorem, $\rho_{l} \rightarrow \rho$ uniformly on $\bar{\Omega}$.
 
 Then we need to prove that $\mathcal{R} = \{\rho(x)x;~x\in \bar{\Omega}\}$ is a near field refractor.
 
 Indeed, fix $x_{0} \in \bar{\Omega}$, then there exist $P_{l}$ and $b_{l}$, such that $\mathcal{E}(P_{l},b_{l})$ supports $\mathcal{R}_{l}$ at $\rho(x_{0})x_{0}$, namely
  $$\rho_{l}(x_{0}) = h(x_{0},P_{l},b_{l}) \quad \text{and} \quad \rho_{l}(x) \leq h(x,P_{l},b_{l}) \quad \text{for all}~x \in \bar{\Omega}.$$
  From (\ref{11}), we have
  $$ x \cdot P_{l} - \frac{\vert P_{l} \vert^{2}}{2C_{1}} \leq b_{l} \leq 2C_{2} - \vert P_{l} \vert,$$
  then there exist $P_{0}$ and $b_{0}$, such that $P_{l} \rightarrow P_{0}$ and $b_{l} \rightarrow b_{0}$ as $l \rightarrow \infty$. 
  
  Now we need to verify that $\mathcal{E}(P_{0},b_{0})$ supports $\mathcal{R}$ at $\rho(x_{0})x_{0}$. For we have
  \begin{equation*}
    \rho(x_{0}) = \lim\limits_{l\rightarrow \infty}\rho_{l}(x_{0}) = \lim\limits_{l\rightarrow \infty}h(x_{0},P_{l},b_{l}) = h(x_{0},P_{0},b_{0})
  \end{equation*}
  and
  $$\rho(x) = \lim\limits_{l\rightarrow \infty}\rho_{l}(x) \leq \lim\limits_{l\rightarrow \infty}h(x,P_{l},b_{l}) = h(x,P_{0},b_{0}) \quad \text{for any}~ x \in \bar{\Omega}.$$
  Hence $\mathcal{E}(P_{0},b_{0})$ supports $\mathcal{R}$ at $\rho(x_{0})x_{0}$.
  
  Finally, we need to prove that $G_{\mathcal{R}_{l}}(F) = \displaystyle\int_{\mathcal{T}_{\mathcal{R}_{l}}(F)}f(x)dx$ converges to $G_{\mathcal{R}}(F) = \displaystyle\int_{\mathcal{T}_{\mathcal{R}}(F)}f(x)dx$ weakly as $l \rightarrow \infty$. Namely, $\varlimsup\limits_{l \rightarrow \infty}G_{\mathcal{R}_{l}}(F) \leq G_{\mathcal{R}}(F)$ holds for any compact sets $F \subset \bar{D}$ and $G_{\mathcal{R}}(U) \leq \varliminf\limits_{l \rightarrow \infty}G_{\mathcal{R}_{l}}(U)$ holds for any open sets $U \subset \bar{D}$.
  
  Indeed, for any compact set $F$, we have
  \begin{equation*}
    \begin{aligned}
    \varlimsup_{l \rightarrow \infty}G_{\mathcal{R}_{l}}(F) & = \varlimsup_{l \rightarrow \infty}\int_{\mathcal{T}_{\mathcal{R}_{l}}(F)}f(x)dx \\
    & \leq \int_{\bar{\Omega}}\varlimsup_{l \rightarrow \infty}\chi_{\mathcal{T}_{\mathcal{R}_{l}}(F)}(x)f(x)dx \\
    & \leq \int_{\mathcal{T}_{\mathcal{R}}(F)}f(x)dx = G_{\mathcal{R}}(F).
    \end{aligned}
  \end{equation*}
  For any open set $U$, we have
  \begin{equation*}
    \begin{aligned}
    G_{\mathcal{R}}(U) & = \int_{\mathcal{T}_{\mathcal{R}}(U)}f(x)dx \\
    & \leq \int_{\bar{\Omega}}\varliminf_{l \rightarrow \infty}\chi_{\mathcal{T}_{\mathcal{R}_{l}}(U)}(x)f(x)dx \\
    & \leq \varliminf_{l \rightarrow \infty}\int_{\bar{\Omega}}\chi_{\mathcal{T}_{\mathcal{R}_{l}}(U)}(x)f(x)dx \\
    & = \varliminf_{l \rightarrow \infty}G_{\mathcal{R}_{l}}(U).
  \end{aligned}
  \end{equation*}
  So we have $G_{\mathcal{R}_{l}}(F) \rightarrow G_{\mathcal{R}}(F)$ weakly as $l \rightarrow \infty$.
  
  For we also have $\mu_{l} \rightarrow \mu$ weakly as $l \rightarrow \infty$, so $\mu(F) = G_{\mathcal{R}}(F)$, hence the weak solution to the near field refraction problem for the critical case exists.
\end{proof}

\section{Conclusion}
\label{sec6}

In this paper, we studied the near field refraction problem for the critical case in negative refractive index material. Snell's law in vector form was first reviewed, and the physical constraints governing refraction were discussed. The refractor for the critical case was then defined, and several of its important properties were analyzed. Using these results, we introduced the notion of the weak solution to the near field refraction problem for the critical case and finally proved the existence of the weak solution. We note that when $\kappa = −1$, the reflected and refracted rays are collinear and oppositely directed, hence it suffices to consider the near field refraction problem, and the far field case needs no separate treatment.

This paper used the Minkowski method to solve the near field refraction problem for the critical case in negative refractive index material. Minkowski method is an iterative approach for solving refraction and reflection problems in geometric optics, which is effective in solving near field refraction problem. The existence of the weak solution to the near field refraction problem for the case $\kappa < -1$ and $-1 < \kappa < 0$ can be proved by using generated Jacobian equation \cite{Tr14}. Can generated Jacobian equation be used to prove the existence of the weak solution to the near field refraction problem for the critical case is still an open problem.

\vspace{3mm}

\noindent {\bf Conflict of Interest}  { The authors declare no conflict of interest.}

\vspace{3mm}

\noindent {\bf Acknowledgments}  { This work was supported by the National Natural Science Foundation of China (No. 12271093), the Jiangsu Provincial Scientific Research Center of Applied Mathematics (No. BK20233002) and Shanghai Institute for Mathematics and Interdisciplinary Sciences (SIMIS) under grant number SIMIS-ID-2025-AD.


\begin{thebibliography}{99}\addtolength{\itemsep}{-1.5ex}\footnotesize

\bibitem{CT05} S. Collardey, A.C. Tarot, P. Pouliguen, K. Mahdjoubi. Use of electromagnetic band‐gap materials for RCS reduction. Microw Opt. Techn. Let. 2005;44(6):546-550.

\bibitem{ET02} S. Enoch, G. Tayeb, P. Sabouroux, N. Guérin, P. Vincent. A metamaterial for directive emission. Phys. Rev. Lett. 2002;89(21):213902.

\bibitem{GH14} C.E. Gutiérrez, Q. Huang. The near field refractor. Ann Inst H Poincaré Anal Non Linéaire 2014;31(4):655-684.

\bibitem{GM21} C.E. Gutiérrez, H. Mawi. On the numerical solution of the near field refractor problem. Appl Math Optim 2021;84(Suppl 2):1877-1902.

\bibitem{GS15} C.E. Gutiérrez, E. Stachura. Uniform refraction in negative refractive index materials. J Opt Soc Am A 2015;32(11):2110-2122.

\bibitem{GS16} C.E. Gutiérrez, E. Stachura. Metamaterial lens design. J Opt Soc Am A 2016;33(10):2020-2026.

\bibitem{GT15} C.E. Gutiérrez, F. Tournier. Regularity for the near field parallel refractor and reflector problems. Calc Var Partial Differential Equations 2015;54(1):917-949.
    
\bibitem{GT21} C.E. Gutiérrez, F. Tournier. $C^{1,\alpha}$-estimates for the near field refractor. Ann Inst H Poincaré Anal Non Linéaire 2021;38(3):577-600.

\bibitem{Ji23} F. Jiang, Weak solutions of generated Jacobian equations, Math. Eng. 2023,5(3):1-20.

\bibitem{JT14} F. Jiang, N.S. Trudinger. On Pogorelov estimates in optimal transportation and geometric optics, Bull. Math. Sci. 2014,4(3): 407-431.

\bibitem{JT18} F. Jiang, N.S. Trudinger. On the second boundary value problem for Monge–Ampère type equations and geometric optics, Arch. Ration. Mech. Anal. 2018,229: 547-567.

\bibitem{JS26} F. Jiang, H. Sui. Near field refraction problem with loss of energy in negative refractive index material, Preprint arXiv 2603.10483v1.

\bibitem{LH04} L. Liu, S. He. Near-field optical storage system using a solid immersion lens with a left-handed material slab. Opt. Express 2004;12(20):4835-4840.

\bibitem{LW21} J. Liu, X.J. Wang, Light refraction is nonlinear optimisation. J. Math. Study 2021;54(2):142-163.

\bibitem{PR02} J.B. Pendry, S.A. Ramakrishna. Near-field lenses in two dimensions. J. Phys-condens. Mat. 2002;14(36):8463.

\bibitem{Pen00} J.B. Pendry. Negative refraction makes a perfect lens. Phys. Rev. Lett. 2000;85(18):3966.

\bibitem{SP00} D.R. Smith, W.J. Padilla, D.C. Vier, S.C. Nemat-Nasser, S. Schultz. Composite medium with simultaneously negative permeability and permittivity. Phys Rev Lett 2000;84(18):4184.

\bibitem{SS01} R.A. Shelby, D.R. Smith, S. Schultz. Experimental verification of a negative index of refraction. Science 2001;292(5514):77-79.

\bibitem{St17} E. Stachura. Existence of weak solutions to refraction problems in Negative Refractive Index Materials. Nonlinear Anal 2017;157:76-103.

\bibitem{Tr14} N.S. Trudinger. On the local theory of prescribed Jacobian equations. Discrete Contin. Dyn. Syst. 2014;34(4):1663-1681.

\bibitem{Tr21} N.S. Trudinger. On the local theory of prescribed Jacobian equations revisited. Math. Eng. 2021;3(6):1-17.

\bibitem{Ve68} V.G. Veselago, The electrodynamics of substances with simultaneously negative values of $\epsilon$ and $\mu$, Phys. Usp., \textbf{10(4)} (1968), 509-514.
		
		
\end{thebibliography}
\end{document}